\documentclass[journal, 10pt]{IEEEtran}

\usepackage{graphicx}

\usepackage{float}
\usepackage{seqsplit}
\usepackage{amsmath,amssymb}
\usepackage{graphicx}
\usepackage{setspace}
\usepackage{subfig}
\usepackage[hyphens]{url}
\usepackage[hyphenbreaks]{breakurl}
\usepackage{cite}

\usepackage{algorithm}
\usepackage[noend]{algpseudocode}

\makeatletter
\def\BState{\State\hskip-\ALG@thistlm}
\makeatother

\newcommand{\norm}[1]{\left\lVert#1\right\rVert}

\newtheorem{theorem}{Theorem}[section]
\newtheorem{lemma}[theorem]{Lemma}
\newtheorem{proposition}[theorem]{Proposition}

\newenvironment{proof}[1][\indent \textit{Proof:}]{\begin{trivlist}
\item[\hskip \labelsep #1]}{\end{trivlist}}
\newenvironment{definition}[1][\indent \textit{Definition:}]{\begin{trivlist}
\item[\hskip \labelsep #1]}{\end{trivlist}}

\newenvironment{remark}[1][\indent \textit{Remark}]{\begin{trivlist}
\item[\hskip \labelsep #1]}{\end{trivlist}}

\newcommand{\qed}{\nobreak \ifvmode \relax \else
      \ifdim\lastskip<1.5em \hskip-\lastskip
      \hskip1.5em plus0em minus0.5em \fi \nobreak
      \vrule height0.75em width0.5em depth0.25em\fi}

\newcommand{\bs}{\boldsymbol}
\newcommand{\ol}{\overline}
\newcommand{\mc}{\mathcal}

\begin{document}

\title{Polynomial-Time Methods to Solve Unimodular Quadratic Programs With Performance Guarantees}

\author{Shankarachary~Ragi,~\IEEEmembership{Member,~IEEE,}
        Edwin~K.~P.~Chong,~\IEEEmembership{Fellow,~IEEE,}
        and~Hans~D.~Mittelmann
\thanks{The work of S.~Ragi and H.~D.~Mittelmann was supported in part by Air
Force Office of Scientific Research under grant FA 9550-15-1-0351. The work of E.~K.~P.~Chong was supported in part by National Science
Foundation under grant CCF-1422658. This work appeared in part in the
\emph{Proceedings of the IEEE International Conference on Acoustics, Speech, and
Signal Processing} (ICASSP), Mar 2017.}
\thanks{S. Ragi and H. D. Mittelmann are with the School of Mathematical and Statistical Sciences, Arizona State University, Tempe,
AZ, 85287, USA e-mail: \{shankarachary.ragi, mittelmann\}@asu.edu.}
\thanks{E. K. P. Chong is with the Department of Electrical and Computer Engineering, Colorado State University, Fort Collins, CO, 80523, USA e-mail: edwin.chong@colostate.edu.}}

\maketitle

\begin{abstract}
We develop polynomial-time heuristic methods to solve unimodular quadratic programs (UQPs) approximately, which are known to be NP-hard. In the UQP framework, we maximize a quadratic function of a vector of complex variables with unit modulus. Several problems in active sensing and wireless communication applications boil down to UQP. With this motivation, we present three new heuristic methods with polynomial-time complexity to solve the UQP approximately. The first method is called \textit{dominant-eigenvector-matching}; here the solution is picked that matches the complex arguments of the dominant eigenvector of the Hermitian matrix in the UQP formulation. We also provide a performance guarantee for this method. The second method, a \textit{greedy} strategy, is shown to provide a performance guarantee of $(1-1/e)$ with respect to the optimal objective value given that the objective function possesses a property called \textit{string submodularity}. The third heuristic method is called \emph{row-swap greedy strategy}, which is an extension to the \emph{greedy strategy} and utilizes certain properties of the UQP to provide a better performance than the \emph{greedy strategy} at the expense of an increase in computational complexity. We present numerical results to demonstrate the performance of these heuristic methods, and also compare the performance of these methods against a standard heuristic method called \emph{semidefinite relaxation}.
\end{abstract}


\IEEEpeerreviewmaketitle

\section{INTRODUCTION}
  Unimodular quadratic programming (UQP) appears naturally in radar waveform-design, wireless communication, and active sensing applications \cite{Mojtaba_UQP}. 
  To state the UQP problem in simple terms- a finite sequence of complex variables with unit
  modulus to be optimized maximizing a quadratic objective function. In the context of a radar system that transmits a linearly encoded burst of pulses, 
  the authors of \cite{Mojtaba_UQP} showed that the problems of designing the coefficients 
  (or codes) that maximize the signal-to-noise ratio (SNR) \cite{Maio_code_design} or minimize the Cramer-Rao lower bound (CRLB) lead to a UQP 
  (see \cite{Mojtaba_UQP,Maio_code_design} for more details). We also know that UQP is NP-hard from the arguments presented in 
  \cite{Mojtaba_UQP,Zhang_NPhard} and the references therein. In this study, we focus on developing tractable heuristic methods to solve the 
  UQP problem approximately having polynomial complexity with respect to the size of the problem. We also provide performance bounds for these heuristic methods. 
  
  In this study, a bold uppercase letter represents a matrix and a bold lowercase letter represents a vector, and if not bold it represents
  a scalar. Let $\bs{s}$ represent the unimodular code sequence of length $N$, where each element of this vector lies on the unit circle $\Omega$
  centered at the origin in the complex plane, i.e., $\Omega = \{ x \in \mathbb{C}, |x| = 1 \}$. The UQP problem is stated as follows:
 
  \begin{equation}\label{UQP}
  \begin{aligned}
  & \underset{\bs{s} \in \Omega^N}{\text{maximize}}
  & & \bs{s}^H \bs{R} \bs{s},
  \end{aligned}
  \end{equation}
  where $\bs{R} \in \mathbb{C}^{N\times N}$ is a given Hermitian matrix. 
  
  There were several attempts at solving the UQP problem (or a variant) approximately or exactly in the past; see references in \cite{Mojtaba_UQP}. For instance, 
  the authors of \cite{Kyril} studied the discrete version of the UQP problem, where the unimodular codes to be optimized are selected from a finite set of points on the complex unit circle around the origin, as opposed to the set of
  all points that lie on this unit circle in our UQP formulation (as shown in (\ref{UQP})). Under the condition that the Hermitian matrix in 
  this discretized UQP is rank-deficient and the rank behaves like $\mathcal{O}(1)$ with respect to the dimension of the problem, the authors of 
  \cite{Kyril} proposed a polynomial time algorithm to obtain the optimal solution. Inspired by these efforts, we propose three new heuristic methods to solve the 
  UQP problem (\ref{UQP}) approximately, where the computational complexity grows only polynomially with the size of the problem. In our study, we exploit certain properties of Hermitian matrices to derive performance bounds for these methods.   
  
  The rest of the paper is organized as follows. In Section~\ref{sec:dom-eig_heur}, we present a heuristic method called \textit{dominant-eigenvector-matching} and a performance bound for this method. In Section~\ref{sec:greedy_heur}, we develop a \textit{greedy} strategy to solve the UQP problem approximately, which has polynomial complexity with respect to the size of the problem; we also derive a performance bound (when $\bs{R}$ satisfies certain conditions) for this method for a class of UQP problems. In Section~\ref{sec:rowswap_greedy}, we discuss the third heuristic method called \emph{row-swap greedy strategy}, and we also derive a performance bound for this method for certain class of UQP problems. In Section~\ref{sec:applic_examples}, we show application examples where our \emph{greedy} and \emph{row-swap greedy} methods are guaranteed to provide the above-mentioned performance guarantees. In Section~\ref{sec:Simulation}, we demonstrate the effectiveness of the above-mentioned heuristic methods via a numerical study. Section~\ref{sec:conclusion} provides a summary of the results and the concluding remarks.
  
  \section{DOMINANT EIGENVECTOR-MATCHING HEURISTIC}\label{sec:dom-eig_heur}
  Let $\lambda_1, \ldots,\lambda_N$ be the eigenvalues of $\bs{R}$ such that $\lambda_1 \leq \cdots \leq \lambda_N$. We can verify that
  \begin{equation*}
  \begin{aligned}
  \lambda_1 N \leq \underset{\bs{s} \in \Omega^N}{\text{max}} \enspace \bs{s^H R s} \leq \lambda_N N.   
  \end{aligned}
  \end{equation*} 
  The above upper bound on the optimal solution ($\lambda_N N$) will be used in the following discussions.
  \begin{definition} 
   In this study, a complex vector $\bs{a}$ is said to be \emph{matching} a 
  complex vector $\bs{b}$ when $\arg(\bs{a}(i)) = \arg(\bs{b}(i))$ for all $i$, where 
  $\bs{a}(i)$ and $\bs{b}(i)$ are the $i$th elements of the vectors $\bs{a}$ and $\bs{b}$ respectively, and $\arg(x)$ represents the argument
  of a complex variable $x$. 
  \end{definition}

  Without loss of generality, we assume that $\bs{R}$ is positive semi-definite. If $\bs{R}$ is not positive semi-definite, 
  we can turn it into one with diagonal loading technique without changing the optimal solution to UQP, i.e., we do the following $\bs{R} = \bs{R} - \lambda_1 \mathbb{I}_{N}$,
  where $\lambda_1$ ($< 0$ as $\bs{R}$ is not semi-definite) is the smallest eigenvalue of $\bs{R}$. Let $\bs{R}$ be diagonalized as follows: $\bs{R} = \bs{U \Lambda U^H}$, where $\bs{\Lambda}$ is a diagonal matrix with eigenvalues $(\lambda_1,\ldots,\lambda_N)$ of $\bs{R}$ as the diagonal elements, and $\bs{U}$ is a unitary matrix 
  with the corresponding eigenvectors as its columns. Let $\bs{U} = [\bs{e}_1 \ldots \bs{e}_N]$, where $\bs{e}_i$ is the eigenvector corresponding to the eigenvalue 
  $\lambda_i$. Thus, the UQP expression can be written as: $\bs{s^H R s} = \bs{s^H U \Lambda U^H s} = \sum_{i=1}^N \lambda_i |t(i)|^2$, 
  where $t(i)$ is the $i$th element of $\bs{U^Hs}$, and $| . |$ is the modulus of a complex number. We know that $\sum_{i=1}^N |t(i)|^2 = N$ for all $s \in \Omega^N$. Ideally, the UQP objective function would be maximum if we can find an $\bs{s}$ such that $|t(N)|^2 = N$ and $|t(i)| = 0$ for all $i<N$; but for any given $\bs{R}$ such an $\bs{s}$ may not exist. Inspired by the above observation, we present the following heuristic method to solve the UQP problem approximately. 
  
  We choose an $\bs{s} \in \Omega^N$ that maximizes the last term in the above summation $|t(N)|$. In other words, we choose an $\bs{s} \in \Omega^N$ that ``matches'' (see the definition presented earlier) $\bs{e}_N$---the dominant eigenvector of $\bs{R}$. We call this method \textit{dominant-eigenvector-matching}. But $\bs{e}_N$ may contain zero elements, and when this happens we set the corresponding entry in the solution vector to $e^{j0}$. The following proposition provides a performance guarantee for this method. Hereafter, this heuristic method is represented by $\mathcal{D}$. The following result provides a performance bound for $\mc{D}$.    
  \begin{proposition}\label{theorem:heur_perf}
   Given a Hermitian and positive semi-definite matrix $\bs{R}$, if $V_{\mathcal{D}}$ and $V_{opt}$ represent the objective function values from the heuristic method 
   $\mathcal{D}$ and the optimal solution respectively for the UQP problem, then
   \begin{equation*}
   \frac{V_{\mathcal{D}}}{V_{opt}} \geq \frac{\lambda_N + (N-1)\lambda_1}{\lambda_N N},
  \end{equation*} 
  where $\lambda_1$ and $\lambda_N$ are the smallest and the largest eigenvalues of $\bs{R}$ of size $N$.
  \end{proposition}
 
  \begin{proof}
  Let $\bs{d}$ be the solution obtained from the heuristic algorithm $\mathcal{D}$. Therefore, the objective function value from $\mathcal{D}$ is
  \begin{equation*}
     V_{\mathcal{D}} =  \bs{d}^H \bs{R} \bs{d} = \sum_{i=1}^N \lambda_i |\bs{e}_i^H \bs{d}|^2,
  \end{equation*}
  where $\lambda_1,\ldots,\lambda_N$ ($\lambda_1 \leq \ldots \leq \lambda_N$) are the eigenvalues of $\bs{R}$ with $\bs{e}_1,\ldots,\bs{e}_N$ being the 
  corresponding eigenvectors. Since $\bs{d}$ matches the dominant eigenvector of $\bs{R}$, we know that 
  \begin{equation*}
  |\bs{e}_N^H \bs{d}|^2 = \left( \sum_{i=1}^N |\bs{e}_N(i)|\right)^2 \geq \sum_{i=1}^N |\bs{e}_N(i)|^2 = 1.
  \end{equation*}
  We know that 
  \begin{equation*}
  \sum_{i=1}^N |\bs{e}_i^H \bs{d}|^2  = \norm{\bs{U}^H\bs{d}}_2^2 = \norm{\bs{d}}_2^2 = N,                  
  \end{equation*}
  where $\norm{\cdot}_2$ represents the 2-norm. Thus,
  \begin{equation*}
  \begin{aligned}
   \frac{V_{\mathcal{D}}}{V_{opt}} &\geq \frac{\lambda_N(|\bs{e}_N^H \bs{d}|^2) + \lambda_1(N-|\bs{e}_N^H \bs{d}|^2)}{\lambda_N N}\\
   &\geq \frac{\lambda_N + (N-1)\lambda_1}{\lambda_N N}.
  \end{aligned}
  \end{equation*}
  \qed
  \end{proof}
  
  The above heuristic method has polynomial complexity as most eigenvalue algorithms (to find the dominant eigenvector) have a computational complexity of at most $\mathcal{O}(N^3)$ \cite{Victor_EigenProb}, e.g., QR algorithm. 
  
  \section{GREEDY STRATEGY}\label{sec:greedy_heur}
  In this section, we present the second heuristic method, which is a \textit{greedy} strategy, and has polynomial-time complexity (with respect to $N$). We also explore the possibility of our objective function possessing a property called string submodularity \cite{Zhenliang_CDC,Zhenliang_IEEEAuto}, which allows our \textit{greedy} method to exhibit a performance guarantee of $(1-1/e)$. First, we describe the \textit{greedy} method, and then explore the possibility of our objective function being string-submodular. Let $\bs{g}$ represent the solution from this greedy strategy, which is obtained iteratively as follows:
  \begin{equation}\label{eq:greedy}
  \begin{aligned}
    \bs{g}(k+1) &= \arg\max_{x\in\Omega} \enspace [\bs{g}(1), \ldots, \bs{g}(k), x]^H \bs{R}_k\\
    & [\bs{g}(1), \ldots, \bs{g}(k), x], \\ 
    &k=1,\ldots,N-1
   \end{aligned}
  \end{equation}
  where $\bs{g}(k)$ is the $k$th element of $\bs{g}$ with $\bs{g}(1) = 1$. In the above expression, $[a,b]$ represents a column vector with elements $a$ and $b$, and $\bs{R}_k$ is the principle sub-matrix of $\bs{R}$ obtained by retaining the first $k$ rows and the first $k$ columns of $\bs{R}$. This method is also described in Algorithm~\ref{alg:greedy}. 
  \begin{algorithm}
  \caption{Greedy Method}\label{alg:greedy}
  \begin{algorithmic}[1]

  \State $\bs{g}(1) \gets 1$ (initialization)
  \State $k \gets 2$
  \State $N \gets \text{length of the code sequence}$ 

  \BState \emph{loop}:
  \If {$k \leq N$}
  \State {$\begin{aligned} \bs{g}(k) \gets \arg\max_{x\in\Omega}  \enspace [\bs{g}(1), \ldots, & \bs{g}(k), x]^H \bs{R}_k \\ & [\bs{g}(1), \ldots, \bs{g}(k), x] \end{aligned}$}
  \State $k \gets k+1$
  \EndIf

  \State \textbf{goto} \emph{loop}.

  \end{algorithmic}
  \end{algorithm}
  
  In other words, we optimize the unimodular sequence element-wise with a partitioned representation of the objective function as shown in (\ref{eq:greedy}), which suggests that the computational complexity grows as $\mathcal{O}(N)$. Let this heuristic method be represented by $\mathcal{G}$. 
  
  The greedy method $\mathcal{G}$ is known to exhibit a performance guarantee of $(1-1/e)$ when the objective function possesses a property called string-submodularity \cite{Zhenliang_CDC,Zhenliang_IEEEAuto,Nemhauser}. To verify if our objective function has this property, we need to re-formulate our problem, which requires certain definitions as described below.
  
  We define a set $A^*$ that contains all possible unimodular strings (finite sequences) of length up to $N$, i.e.,
  \begin{equation*}
  \begin{aligned}
   A^* &= \{(s_1,\ldots,s_k)| s_i \in \Omega\; \text{for}\; \\ 
   &i=1,\ldots,k\; \text{and} \;k=1,\ldots,N\},
   \end{aligned}
  \end{equation*}
  where $\Omega = \{ x \in \mathbb{C}, |x| = 1 \}$. Notice that all the unimodular sequences of length $N$ in the UQP problem are elements in the set $A^*$. 
  For any given Hermitian matrix $\bs{R}$ of size $N$, let $f:A^*\rightarrow \mathbb{R}$ be a quadratic function defined as $f(A) = A^H \bs{R}_k A$,
  where $A = (s_1,\ldots,s_k) \in A^*$ for any $1 \leq k \leq N$, and $\bs{R}_k$ is the principle sub-matrix of $\bs{R}$ of size $k\times k$ as defined before. We represent string concatenation by $\oplus$, i.e., if $A = (a_1,\ldots,a_k) \in A^*$ and $B = (b_1,\ldots,b_r) \in A^*$ for any 
  $k+r \leq N$, then $A\oplus B = (a_1,\ldots,a_k,b_1,\ldots,b_r)$. A string $B$ is said to be contained in $A$, represented by $B \preceq A$ if there exists a $D \in A^*$ 
  such that $A = B\oplus D$. For any $A,B \in A^*$ such that $B \preceq A$, a function $f:A^* \rightarrow \mathbb{R}$ is said to be string-submodular \cite{Zhenliang_CDC,Zhenliang_IEEEAuto} if both the following conditions are true:
   \begin{enumerate}
    \item $f$ is \emph{forward monotone}, i.e., $f(B) \leq f(A)$.
    \item $f$ has the diminishing-returns property, i.e., $f(B\oplus (a)) - f(B) \geq f(A \oplus (a)) - f(A)$ for any $a \in \Omega$.
   \end{enumerate}
  
  Now, going back to the original UQP problem, the UQP quadratic function may not be a string-submodular function for any given Hermitian matrix $\bs{R}$. However, without loss of generality, we will show that we can transform $\bs{R}$ to $\ol{\bs{R}}$ (by manipulating the diagonal entries) such that the resulting quadratic 
  function $A_k^H \ol{\bs{R}}_k A_k$ for any $1\leq k\leq N$ and $A_k \in A^*$ is string-submodular, where $\ol{\bs{R}}_k$ is the principle sub-matrix of 
  $\ol{\bs{R}}$ of size $k\times k$ as defined before. The following algorithm shows a method to transform $\bs{R}$ to such a $\ol{\bs{R}}$ that induces string-submodularity on the UQP problem.
  \begin{enumerate}
   \item First define $\delta_1,\ldots,\delta_N$ as follows:
   \begin{equation}\label{eq:deltas}
    \delta_k = \sum_{i=1}^{k-1} |r_{ki}|,
   \end{equation}
   where $k= 2,\ldots,N$, $\delta_1 = 0$, and $|r_{ki}|$ is the modulus of the entry in the $k$th row and the $i$th column of $\bs{R}$.   
  \item Define a vector with $N$ entries $(a_1,\ldots,a_N)$, where $a_k = 2\delta_k + 4\left(\sum_{i=1}^{N-k} \delta_{k+i}\right)$ for $k=1,\ldots,N-1$, and $a_N = 2\delta_N$. 
  \item Define $\ol{\bs{R}}$ as follows:
  \begin{equation}\label{eq:Rbar_def}
   \ol{\bs{R}} = \bs{R} - Diag(\bs{R}) + diag((a_1,\ldots,a_N)),
  \end{equation}
  where $Diag(\bs{R})$ is a diagonal matrix with diagonal entries the same as those of $\bs{R}$ in the same order, and $diag((a_1,\ldots,a_N))$ is a diagonal matrix with diagonal
  entries equal to the array $(a_1,\ldots,a_N)$ in the same order. 
  \end{enumerate}
  Since we only manipulate the diagonal entries of $\bs{R}$ to derive $\ol{\bs{R}}$, the following is true:   
  \begin{equation*}
   \arg\max_{A_N\in \Omega^N} \enspace A_N^H \ol{\bs{R}} A_N = \arg\max_{A_N \in \Omega^N} \enspace A_N^H \bs{R} A_N.
  \end{equation*}
  
  For any given Hermitian matrix $\bs{R}$ and the derived $\ol{\bs{R}}$ (as shown above), let $F:A^*\rightarrow \mathbb{R}$ be defined as 
  \begin{equation}\label{eq:obj_Rbar}
   F(A_k) = A_k^H \ol{\bs{R}}_k A_k,
  \end{equation}
  where $A_k \in A^*$. 
  
  \begin{lemma}\label{lemma:dim_ret_bnds}
  For a given $\bs{R}$ and $F:A^*\rightarrow \mathbb{R}$ as defined in (\ref{eq:obj_Rbar}) with the derived $\ol{\bs{R}}$ from $\bs{R}$, and for any 
  $A,B \in A^*$ such that $B \preceq A$, with $B = (b_1,\ldots,b_k)$ and $A = (b_1,\ldots,b_k,\ldots,b_l)$ ($k\leq l \leq N$), the inequalities
  \small
  \begin{equation*}
   4\sum_{i=k+1}^l \sum_{j=1}^{N-i} \delta_{i+j} \leq F(A) - F(B) \leq 
   4\sum_{i=k+1}^l \left( \delta_i + \sum_{j=1}^{N-i} \delta_{i+j}\right)
  \end{equation*}
  \normalsize
  hold where $\delta_i$ for $i=1,\ldots,N$ are defined in (\ref{eq:deltas}). 
  \end{lemma}
  
  \begin{proof}
  Let $r_{ij}$ be the entry from $\ol{\bs{R}}$ at the $i$th row and $j$th column. From (\ref{eq:Rbar_def}), we can verify that $r_{ii} = a_i$ for $i=1,\ldots,N$. 
  Therefore, from the definitions of $F$ and $a_i$ for $i=1,\ldots,N$, we can verify that 
   \begin{equation*}
   \begin{split}
     F(A) - F(B) & = A^H \ol{\bs{R}}_l A - B^H \ol{\bs{R}}_k B \\ 
		& = \sum_{i=k+1}^l r_{ii}b_i b_i^* + b_i^*\left( \sum_{j=1}^{i-1} r_{ij} b_j \right)\\ 
		& + b_i \left( \sum_{j=1}^{i-1} r_{ji}b_j^*\right)
   \end{split}
   \end{equation*}
  Therefore, from the definitions of $\delta_i$ for $i=1,\ldots,N$ in (\ref{eq:deltas}), it follows that 
  \begin{equation*}
    \begin{aligned} 
     \sum_{i=k+1}^l \left( r_{ii} - 2\left(\sum_{j=1}^{i-1} |r_{ij}| \right) \right) &\leq F(A) - F(B)  \\
     &\leq \sum_{i=k+1}^l \left(r_{ii} + 2\left(\sum_{j=1}^{i-1} |r_{ij}| \right) \right)
    \end{aligned}
   \end{equation*}
   \begin{equation*}
    \begin{aligned} 
     4\sum_{i=k+1}^l \left( \delta_i + \sum_{j=1}^{N-i} \delta_{i+j}\right) &\leq F(A) - F(B) \\ 
     &\leq 4\sum_{i=k+1}^l \left( \delta_i + \sum_{j=1}^{N-i} \delta_{i+j}\right)
    \end{aligned}
   \end{equation*}
  \qed 
  \end{proof}
  
  \begin{lemma}\label{lemma:stringsubmod}
  Given any Hermitian matrix $\bs{R}$ of size $N$, the objective function $F:A^*\rightarrow \mathbb{R}$ defined in (\ref{eq:obj_Rbar}) is string submodular.   
  \end{lemma}
  
  \begin{proof} \underline{Forward monotonocity proof}\\ 
  Let $A,B \in A^*$ such that $B \preceq A$, therefore $A$ and $B$ are of the form $B = (b_1,\ldots,b_k)$ and $A = (b_1,\ldots,b_k,\ldots,b_l)$, where $k \leq l \leq N$. Thus,
  from Lemma~\ref{lemma:dim_ret_bnds}
  \begin{equation*}
    F(A) - F(B) \geq 4 \sum_{i=k+1}^l \sum_{j=1}^{N-i} \delta_{i+j} \geq 0
  \end{equation*}
  as $\delta_i \geq 0$ for all $i=1,\ldots,N$.
  
  \underline{Diminishing returns proof}\\
  For any $u \in \Omega$, $u$ is also an element in the set $A^*$, i.e., $\{u\} \in A^*$. Therefore, from Lemma~\ref{lemma:dim_ret_bnds} the
  following inequalities hold true:
  \begin{equation*}
   \begin{aligned}
    &\left(F(B\oplus\{u\}) - F(B)\right) - \left(F(A\oplus\{u\}) - F(A)\right) \\ 
    &\geq 4\left(\sum_{j=1}^{N-k-1} \delta_{j+k+1}\right) - 4\left( \delta_{l+1} + \sum_{j=1}^{N-l-1} \delta_{l+j+1}\right)\\
    &= 4\left( \sum_{j=k+2}^{l+1} \delta_j\right) - 4\delta_{l+1} \geq 0.
   \end{aligned}
  \end{equation*}
  \qed
  \end{proof}
    
  We know from \cite{Zhenliang_CDC,Zhenliang_IEEEAuto} that the performance of the heuristic method $\mathcal{G}$ is at least $(1-1/e)$ of the optimal value with respect to the function $F$, i.e., if $\bs{g}\in A^*$ is the solution from the heuristic method $\mathcal{G}$ and if $\bs{o}$ is the optimal solution that maximizes the objective function $F$ as in 
  \begin{equation}\label{eq:optseq_rbar}
   \bs{o} = \arg\max_{A_N \in \Omega^N} A_N^H \ol{\bs{R}}A_N,
  \end{equation}
  then 
  \begin{equation}\label{eq:submod_guar}
   F(\bs{g}) \geq \left(1-\frac{1}{e}\right)F(\bs{o}). 
  \end{equation}
  Although we have a performance guarantee for the greedy method with respect to $F$, we are more interested in the performance
  guarantee with respect to the original UQP quadratic function with the given matrix $\bs{R}$. We explore this idea with the following results. 
  
  \begin{remark}\label{remark:RbarExp}
  For any Hermitian matrix $\bs{R}$, if $\ol{\bs{R}}$ is derived from $\bs{R}$ according to (\ref{eq:Rbar_def}), then $Tr(\ol{\bs{R}}) = \sum_{k=2}^N (4k-2)\delta_{k}$, where $\delta_k$ for $k=2,\ldots,N$ are defined in (\ref{eq:deltas}).
  \end{remark}
      
  \begin{theorem}\label{theorem:perf_bound}
  For a given Hermitian matrix $\bs{R}$, if $Tr(\ol{\bs{R}}) \leq Tr(\bs{R})$, where $\ol{\bs{R}}$ is derived from $\bs{R}$ as described earlier in this section, then 
  \begin{equation*}
   \bs{g}^H \bs{R} \bs{g} \geq \left(1-\frac{1}{e}\right) \left( \max_{\bs{s} \in \Omega^N} \enspace \bs{s}^H \bs{R}\bs{s} \right),
  \end{equation*}
  where $\bs{g}$ is the solution from the greedy method $\mathcal{G}$. 
  \end{theorem}
  
  \begin{proof}
  For any $A_k \in A^*$, we know that $F(A_k) = A_k^H \ol{\bs{R}}_k A_k = A_k^H \bs{R}_k A_k + Tr(\ol{\bs{R}}_k) - Tr(\bs{R}_k)$. Since $F$ in 
  (~\ref{eq:Rbar_def}) is string-submodular, given $\bs{g}$ is the solution from the heuristic $\mathcal{G}$ and $\bs{o}$ being the optimal solution that 
  maximizes $F$ over all possible solutions of length $N$, from (\ref{eq:submod_guar}) we know that
  \begin{equation*}
  F(\bs{g}) \geq \left(1 - \frac{1}{e}\right) F(\bs{o})
  \end{equation*}
  
  \begin{equation*}
  \begin{aligned}
  \bs{g}^H \bs{R} \bs{g} &+ Tr(\ol{\bs{R}}) - Tr(\bs{R}) \\
  &\geq \left(1 - \frac{1}{e}\right) \left(\bs{o}^H \bs{R} \bs{o} + Tr(\ol{\bs{R}}) - Tr(\bs{R})\right)\\
  \end{aligned}
  \end{equation*}
  
  \begin{equation}\label{eq:greedy_opt_comp}
  \begin{aligned}
  \bs{g}^H \bs{R} \bs{g} &\geq \left(1 - \frac{1}{e}\right)\bs{o}^H \bs{R} \bs{o} + \frac{1}{e}\left(Tr(\bs{R}) -Tr(\ol{\bs{R}}) \right) \\
  &\geq \left(1 - \frac{1}{e}\right)\bs{o}^H \bs{R} \bs{o}.
  \end{aligned}
  \end{equation}
  \qed
  \end{proof}
  
  We are interested in finding classes of Hermitian matrices that satisfy the requirement $Tr(\ol{\bs{R}}) \leq Tr(\bs{R})$, so that the above result holds true. Intuitively, it may seem that diagonally dominant matrices satisfy the above requirement. But it is easy to find a counter example $\bs{R} = \begin{bmatrix} 2 & 1\\1 & 2\end{bmatrix}$, where $Tr(\ol{\bs{R}}) = 6$ and $Tr(\bs{R}) = 4$. Clearly, diagonal dominance is not sufficient to guarantee that the result in Theorem~\ref{theorem:perf_bound} holds true. Thus, we introduce a new kind of diagonal dominance called $M$-dominance, which lets us finding a class of Hermitian matrices for which the result in Theorem~\ref{theorem:perf_bound} holds true. A square matrix $\bs{R} = [r_{ij}]_{N\times N}$ is said to be $M$-dominant if $r_{ii} \geq M \left(\sum_{j=1, j\neq i}^N |r_{ij}| \right); \forall i$. 
  
  \begin{proposition}\label{prop:2Ndom}
  If a Hermitian matrix $\bs{R}$ of size $N$ is $2N$-dominant, then $Tr(\ol{\bs{R}}) \leq Tr(\bs{R})$, where $\ol{\bs{R}}$ is derived from $\bs{R}$ according to (\ref{eq:Rbar_def}), and
   \begin{equation*}
   \bs{g}^H \bs{R} \bs{g} \geq \left( 1 - \frac{1}{e} + \frac{1}{e}\left(\frac{1}{2N+1}\right)\right) \left( \max_{\bs{s} \in \Omega^N} \enspace \bs{s}^H \bs{R}\bs{s} \right),
  \end{equation*}
  where $\bs{g}$ is the solution from the greedy method $\mathcal{G}$. 
  \end{proposition}
 \begin{proof}
 See Appendix. 
 \end{proof}
  Clearly, if $\bs{R}$ is $2N-$dominant, then the greedy method $\mc{G}$ provides a guarantee of $(1-1/e)$. From the above proposition, it is clear that for such a matrix, $\mc{G}$ can provide a tighter performance bound of $\left( 1 - \frac{1}{e} + \frac{1}{e}\left(\frac{1}{2N+1}\right)\right)$. But this bound quickly converges to $(1-1/e)$ as $N \rightarrow \infty$, as shown in Figure~\ref{fig:bnds_vs_N}. As it turns out, the above result may not have much practical significance, as it requires the matrix $\bs{R}$ to be $2N-$dominant, which narrows down the scope of the result. Moreover, as $N$ increases, the bound looses any significance because the lower bound on the UQP objective value for any solution is much greater than the above derived bound. In other words, if $\bs{R}$ is $2N-$dominant, the lower bound on the performance of any UQP solution $\bs{s}$ is given by $\bs{s}^H \bs{Rs} \geq \left(\frac{2N-1}{2N+1} \right) \bs{o}^H \bs{Rs}$, where $\bs{o}$ is the optimal solution for the UQP. Clearly, for $N > e$ (i.e., $N \geq 3$), $\frac{2N-1}{2N+1} > \left( 1 - \frac{1}{e} + \frac{1}{e}\left(\frac{1}{2N+1}\right)\right)$. Thus, $2N-$dominance requirement may be  a strong condition, and further investigation may be required to look for a weaker condition that satisfies $Tr(\ol{\bs{R}}) \leq Tr(\bs{R})$.
  
In summary, for applications with large $N$, the result in Proposition~\ref{prop:2Ndom} does not hold much significance, as can be seen from Figure~\ref{fig:bnds_vs_N}.      
  \begin{figure}[t]
  \includegraphics[width= \columnwidth, trim = 30 180 25 180,clip]{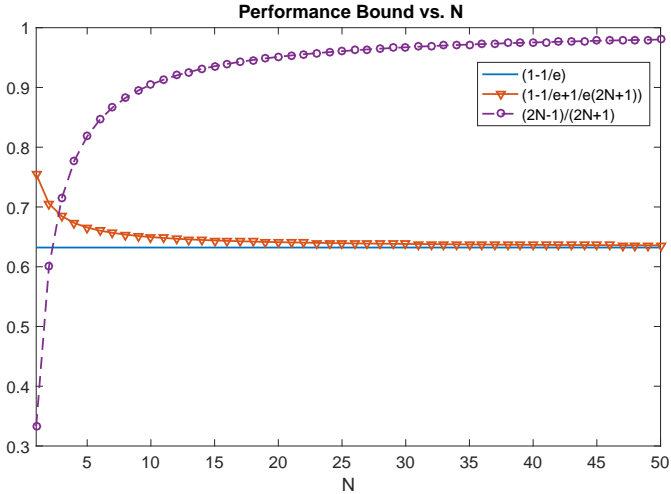}
  \caption{Performance bounds for 2N-dominant matrices vs. N}
  \label{fig:bnds_vs_N}
  \end{figure}   
  
  \section{Row-Swap Greedy Strategy}\label{sec:rowswap_greedy}
  In this section, we present the third heuristic method to solve the UQP approximately. Let $\bs{P}_{mn}$ be a row-switching transformation such that $\bs{P}_{mn}\bs{A}$ swaps the $m$th and $n$th rows of $\bs{A}$ and $\bs{A}\bs{P}_{mn}$ swaps the $m$th and $n$th columns of $\bs{A}$. Let $\mc{P} = \{\bs{P}_{mn} | m=1,\ldots,N; \, n=1,\ldots,N; \, m>n\}$ be a collection of all such matrices that are exactly one row-swap operation away from the identity matrix of size $N \times N$. In the UQP, if we replace $\bs{R}$ with $\bs{P} \bs{R} \bs{P}$ for any $\bs{P} \in \mc{P}$, and if
  \begin{equation}\label{UQP_perm}
  \begin{aligned}
  \hat{\bs{o}} =& \underset{\bs{s} \in \Omega^N}{\text{arg max}}
  & & \bs{s}^H \bs{PRP} \bs{s},
  \end{aligned}
  \end{equation}
  then we can relate $\hat{\bs{o}}$ to the optimal solution of the original UQP (\ref{UQP}) as follows: $\hat{\bs{o}}= \bs{Po}$. We also know that $\bs{o}^H \bs{Ro} = \hat{\bs{o}}^H \bs{R} \hat{\bs{o}}$, i.e., for any row-switching matrix $\bs{P} \in \mc{P}$, the optimal objective value does not change if we replace $\bs{R}$ by $\bs{PRP}$ in the UQP. However, the objective value from the greedy strategy changes if $\bs{R}$ in the UQP is replaced by $\bs{PRP}$. Thus, for a given UQP problem, we may be able to improve the performance of the greedy strategy by simply replacing $\bs{R}$ by $\bs{PRP}$ for some $\bs{P} \in \mc{P}$. We are interested in finding which matrix $\bs{P}$ among $N(N-1)/2$ possible matrices in $\mc{P}$ gives us the best performance from the greedy strategy (note that $|\mc{P}| = N(N-1)/2$). We know that each of the above-mentioned $N(N-1)/2$ objective values (one each from solving the UQP with $\bs{R}$ replaced by $\bs{PRP}$) is upper bounded by the optimal objective value of the original UQP. Clearly, the best performance from the greedy strategy can be obtained by simply picking a matrix from the collection $\mc{P} \bigcup \{\bs{I}_{N}\}$ ($\bs{I}_N$ is the identity matrix of size $N$) that gives maximum objective value. We call this method \emph{row-swap greedy strategy}. The motivation for using this strategy is that one of the $(N(N-1)/2) + 1$ row-switching matrices (including the identity matrix) moves us close to the global optimum. This method is also described in Algorithm~\ref{alg:row-greedy}. 
  \begin{algorithm}
  \caption{Row-Swap Greedy Method}\label{alg:row-greedy}
  \begin{algorithmic}[1]
  \State $\text{Greedy}(\bs{P})$ outputs the objective function value and the solution from the standard greedy method given the row-switching matrix $\bs{P}$, i.e., solves (\ref{UQP_perm})   
  \State $k \gets 1$
  \State $N \gets \text{length of the code sequence}$ 
  \State $M \gets (N(N-1)/2) + 1$ \Comment{Maximum number of row-switching matrices}
  \State $\{\bs{P}_1,\ldots,\bs{P}_M\} \gets \text{set of all possible row-switching matrices}$
  \State $\bs{V} \gets 0$ \Comment{Initialization}

  \BState \emph{loop}:
  \If {$k \leq M$}
  \State $[\bs{V}_{temp},\bs{g}_{temp}] \gets \text{Greedy}(\bs{P}_k)$ 
      \If {$\bs{V}_{temp} > \bs{V}$}
	 \State $\bar{\bs{g}} \gets \bs{g}_{temp}$
	 \State $\bs{V} \gets \bs{V}_{temp}$
      \EndIf
  \State $k \gets k+1$
  \EndIf

  \State \textbf{goto} \emph{loop}
  \State Solution from this method is stored in $\bar{\bs{g}}$
  \end{algorithmic}
  \end{algorithm}
  
  The objective value from the \emph{row-swap greedy method} is given by 
  \begin{equation}\label{eq:row-swap}
  \max_{\bs{P}\in \mc{P} \bigcup \{\bs{I}_{N}\}} V_{\bs{P}},
  \end{equation}
  where $V_{\bs{P}}$ is the objective value from the greedy strategy applied to the UQP with $\bs{R}$ replaced by $\bs{PRP}$. Clearly, the row-swap greedy strategy outperforms the ``greedy strategy" and provides a performance guarantee of $(1-1/e)$ as $\max_{\bs{P}\in \mc{P} \bigcup \{\bs{I}_{N}\}} V_{\bs{P}} \geq V_{\bs{I}_N} \geq (1-1/e) (\bs{o}^H \bs{Ro})$, where $\bs{o}$ is the optimal solution to the UQP. We note that the computational complexity for the \emph{row-swap greedy strategy} grows as $\mc{O}(N^2(N-1)/2)$. 
  
  \begin{remark}
  It is quite possible, but unlikely (confirmed by our numerical study in Section~\ref{sec:Simulation}), that the performance from the \emph{row-swap greedy method} may remain exactly the same as the standard \emph{greedy method}, which happens when row-switching does not improve the performance. In this case, the optimum solution to (\ref{eq:row-swap}) is $V_{\bs{I}_N}$.  
  \end{remark}
   
  \section{Application Examples}\label{sec:applic_examples}
  In the case of a monostatic radar that transmits a linearly encoded burst of pulses (as described in \cite{Mojtaba_UQP}), the problem of optimizing the code elements that maximize the SNR boils down to UQP, where $\mathrm{SNR} = |a|^2 \bs{c}^H \bs{R} \bs{c}$, $\bs{R} = \bs{M}^{-1} \odot (\bs{pp}^H)^*$ ($\odot$ represents the Hadamard product), $\bs{M}$ is an error covariance matrix (of size $N$) corresponding to a zero-mean Gaussian vector, $a$ represents channel propagation and backscattering effects, $\bs{c}$ represents the code elements, and $\bs{p}$ is the temporal steering vector. See \cite{Maio_code_design} for a detailed study of this application problem. From Theorem~\ref{theorem:perf_bound}, we know that if $Tr(\ol{\bs{R}}) \leq Tr(\bs{R})$ for this application, then the \emph{greedy} and the \emph{row-swap greedy} methods for this application are guaranteed to provide the performance of $\left( 1 - \frac{1}{e}\right)$ of that of the optimal.   
  
  In the case of a linear array of $N$ antennas, the problem of estimating the steering vector in adaptive beam-forming boils down to UQP as described in \cite{Mojtaba_UQP} \cite{Khab_AdapBeamForm}, where the objective function is $\bs{c}^H \bs{M}^{-1} \bs{c}$, where $\bs{M}$ is the sample covariance matrix (of size $N$), and $c$ represents the steering vector; see \cite{Khab_AdapBeamForm} for details on this application problem. Again, we can verify that if $Tr(\ol{\bs{R}}) \leq Tr(\bs{R})$, where $\bs{R} = \bs{M}^{-1}$, then the \emph{greedy} and the \emph{row-swap greedy} methods provide a performance guarantee of $\left( 1 - \frac{1}{e}\right)$, as the result in Theorem~\ref{theorem:perf_bound} holds true for this case as well.  
  
  \section{SIMULATION RESULTS}\label{sec:Simulation}
  We test the performance of the heuristic method $\mathcal{D}$ numerically for $N=20, 50, 100$. We generate 500 Hermitian and positive semi-definite matrices randomly for each $N$, 
  and for each matrix we evaluate $V_{\mathcal{D}}$ (value from the method $\mathcal{D}$) and the performance bound derived in Proposition~\ref{theorem:heur_perf}. To generate a random Hermitian and positive semi-definite matrix, we use the following algorithm: 1) first we generate a random Hermitian matrix $A$ using the function \emph{rherm}, which is available at \cite{rherm}; 2) second we replace the eigenvalues of $A$ with values randomly (uniform distribution) drawn from the interval $[0,1000]$. Figure~\ref{fig:Heur_D_perf} shows plots of $\frac{V_{\mathcal{D}}}{\lambda_N N}$ (normalized objective function value) for each $N$ along with the performance bounds for $\mc{D}$, which also shows $\frac{V_{\mathrm{rand}}}{\lambda_N N}$, where $V_{\mathrm{rand}}$ is the objective function value when the solution is picked randomly from $\Omega^N$. The numerical results clearly show that the method $\mathcal{D}$ outperforms (by a good margin) random selection, and more importantly the performance of $\mathcal{D}$ is close to the optimal strategy, which is evident from the simulation results, where the objective function value from $\mathcal{D}$ is at least $90\%$ (on average) of the upper bound on the optimal value for each $N$. The results show that the lower bound is much smaller than the value we obtain from the heuristic method for every sample. In our future study, we will tighten the performance bound for $\mathcal{D}$ as the results clearly show that there is room for improvement. 
  
  \begin{figure}[t]
  \includegraphics[width= \columnwidth, trim = 60 160 130 160,clip]{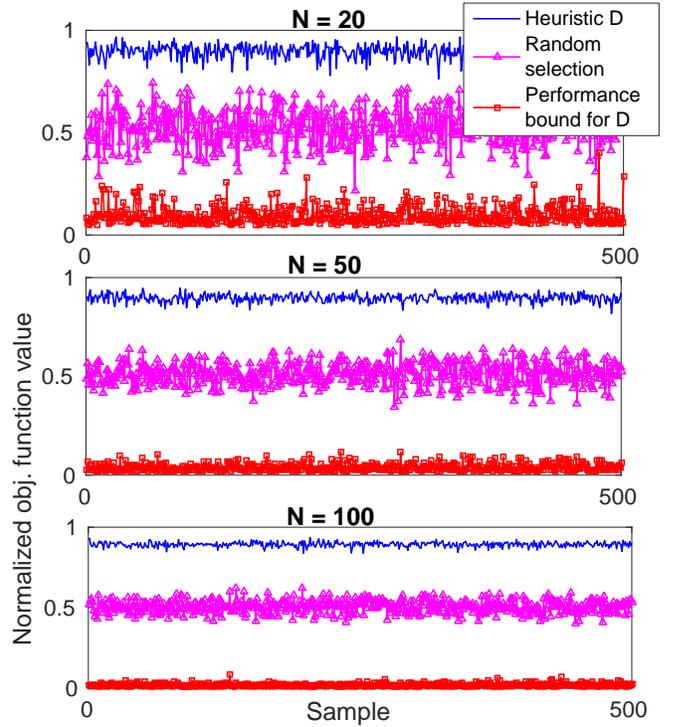}
  \caption{Performance of $\mathcal{D}$ for $N=20, 50, 100$}
  \label{fig:Heur_D_perf}
  \end{figure}    
  
  Figure~\ref{fig:Heur_G_perf} shows the normalized objective function value from the greedy method, for each $N$, along with the bound $(1-1/e)$, supporting the result from Theorem~\ref{theorem:perf_bound}.
  
  \begin{figure}[t]
  \includegraphics[width= \columnwidth, trim = 60 190 135 110,clip]{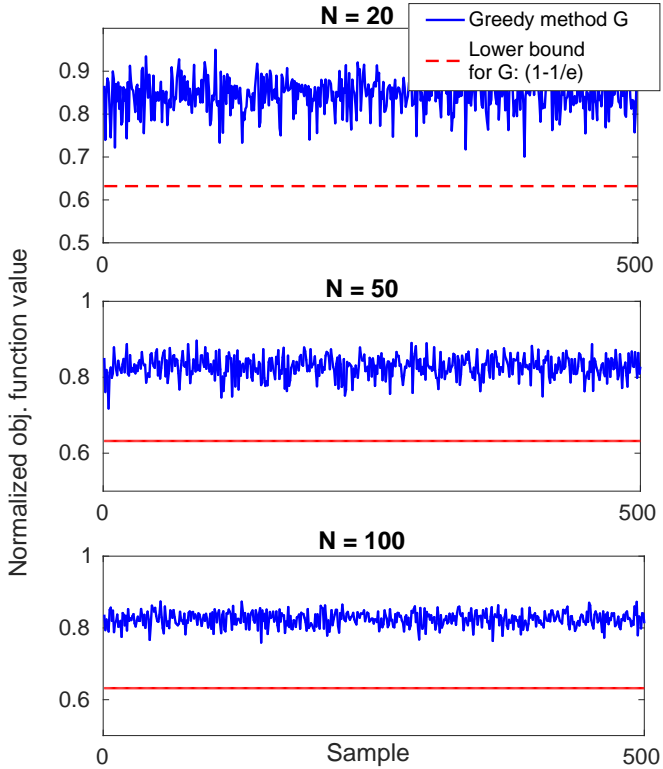}
  \caption{Performance of $\mathcal{G}$ for $N=20, 50, 100$}
  \label{fig:Heur_G_perf}
  \end{figure}
  
  We now present numerical results to show the performance of the \emph{row-swap greedy method} for $N=10,20,30$. We generate 100 Hermitian and positive semi-definite matrices. For each of these matrices, we solve the UQP via the \emph{row-swap greedy method} and also evaluate the performance bound $\left( 1 - \frac{1}{e}\right)$. Figure~\ref{fig:rswapG_perf} shows plots of the normalized objective function values from the \emph{row-swap greedy method} along with the performance bounds. It is evident from these plots that the above heuristic method performs much better than the lower bound suggests, and also suggests that this method performs close to optimal. 
  \begin{figure}[h]
  \includegraphics[width= \columnwidth, trim = 180 250 190 190,clip]{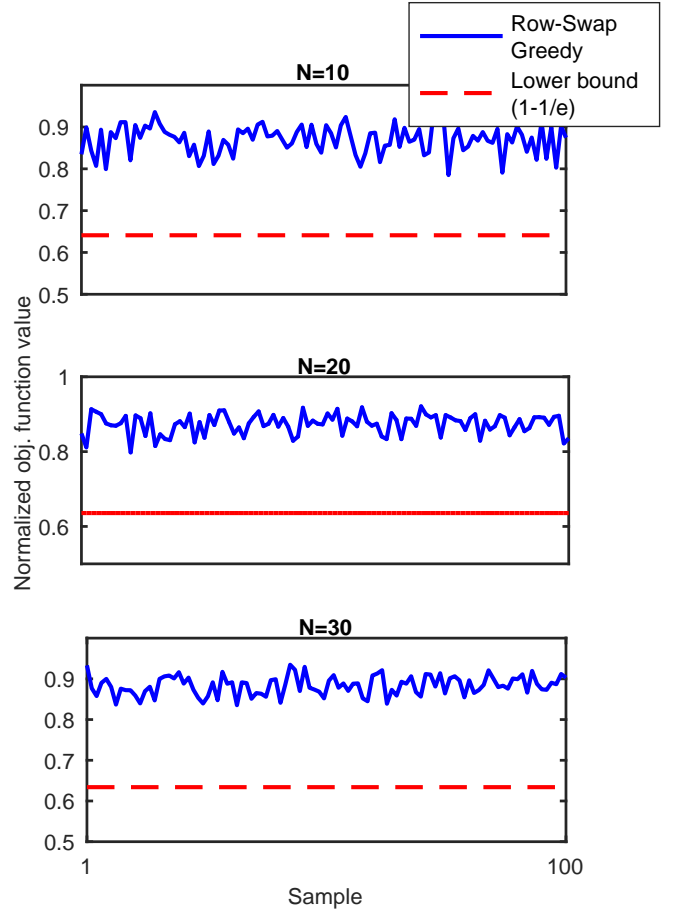}
  \caption{Performance of row-swap greedy method for $N=10,20,30$}
  \label{fig:rswapG_perf}
  \end{figure}
  
  We now compare the performance of the heuristic methods presented in this study against a standard benchmark method called \emph{semidefinite relaxation} (SDR). The following is a brief description of SDR, as described in \cite{Mojtaba_UQP} (repeated here for completeness). We know that $\bs{s}^H \bs{Rs} = \text{tr}(\bs{s}^H \bs{Rs}) = \text{tr}(\bs{Rs}\bs{s}^H)$. Thus, UQP can also be stated as follows:
  \begin{equation*}
  \begin{aligned}
  & \underset{\bs{S} \in \Omega^N}{\text{maximize}} 
  & & \text{tr}(\bs{RS}) \\
  & \text{subject to} 
  & & \bs{S} = \bs{s}\bs{s}^H, \, \bs{s}\in \Omega^N.
  \end{aligned}
  \end{equation*}
  The rank constraint $\bs{S} = \bs{s}\bs{s}^H$ is what makes the UQP hard to solve exactly. If this constraint is relaxed, then the resulting optimization problem is a semidefinite program, as shown below:
  \begin{equation*}
  \begin{aligned}
  & \underset{\bs{S} \in \Omega^N}{\text{maximize}}
  & & \text{tr}(\bs{RS}) \\
  & \text{subject to} 
  & & [\bs{S}]_{k,k} = 1, \, k=1,\ldots,N\\
  &&& \text{$\bs{S}$ is positive semidefinite.} 
  \end{aligned}
  \end{equation*}
  The above method is called \emph{semidefinite relaxation} (SDR). The semidefinite program shown above can be solved in polynomial time by any interior point method \cite{Boyd_CVXopt}; we use a solver called \emph{cvx} \cite{cvx} to solve this SDR. 
  
  The authors of \cite{Mojtaba_UQP} proposed a \emph{power method} to solve the UQP approximately, which is an iterative approach described as follows:
  \begin{equation*}
  \bs{s}^{t+1} = e^{j \mathrm{arg}(\bs{R}\bs{s}^{t})}, 
  \end{equation*}
  where $\bs{s}^0$ is initialized to a random solution in $\Omega^N$. The authors also proved that the objective function value is guaranteed to increase with $t$. 
  
  We now test the performance of our proposed heuristic methods - \emph{dominant eigenvector matching heuristic}, \emph{greedy strategy}, and \emph{row-swap greedy strategy} against existing methods such as the \emph{SDR} and the above-mentioned \emph{power method}. For this purpose, we generate 100 Hermitian and positive-semidefinite matrices. For each of these matrices, we solve the UQP approximately with the above-mentioned heuristic methods. Figure~\ref{fig:all_perf_plot_2} shows the cumulative distribution function of the objective values and the execution times of the heuristic methods for several values of $N$. It is evident from this figure that the proposed heuristic methods significantly outperform the standard benchmark method - SDR. Specifically, the \emph{row-swap greedy} and the \emph{dominant eigenvector-matching} methods deliver the best performance among the methods considered here. However, the \emph{row-swap greedy} method is the most expensive method (in terms of execution time) among the methods considered. Also, we can distinctly arrange a few methods in a sequence of increasing performance (statistical) as follows: \emph{SDR method}, \emph{greedy strategy}, \emph{dominant eigenvector matching heuristic} or \emph{row-swap greedy strategy}. This figure also shows the cumulative distribution functions of the execution times for each $N$, which suggests that all the heuristic methods considered here can be arranged in a sequence of increasing performance (decreasing execution time) as follows: \emph{row-swap greedy strategy}, \emph{SDR method}, \emph{greedy strategy}, \emph{power-method}, and \emph{dominant eigenvector matching heuristic}. 
  \begin{figure*}[t]
  \includegraphics[width= \textwidth, trim = 0 140 0 145,clip]{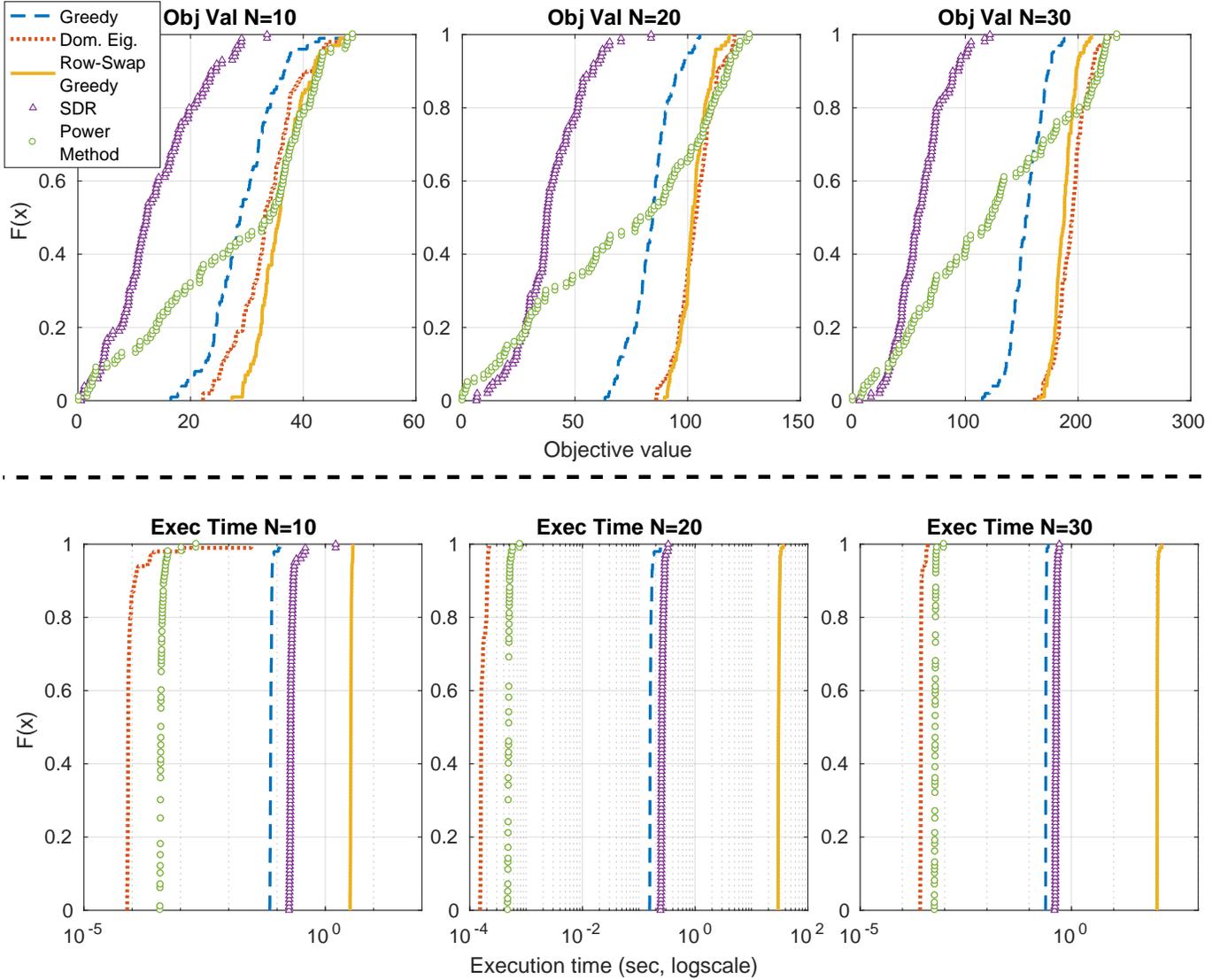}
  \caption{Performance of heuristic methods for $N=10,20,30$}
  \label{fig:all_perf_plot_2}
  \end{figure*}
  
%
  \section{CONCLUDING REMARKS}\label{sec:conclusion} 
  We presented three new heuristic methods to solve the UQP problem approximately with polynomial-time complexity with respect to the size of the problem. The first heuristic method
  was based on the idea of matching the unimodular sequence with the dominant eigenvector of the Hermitian matrix in the UQP formulation. We have provided a performance bound for this heuristic that depends on the eigenvalues of the Hermitian matrix in the UQP. The second heuristic method is a \textit{greedy} strategy. We showed that under loose conditions on the Hermitian matrix, the objective function would possess a property called string submodularity, which then allowed this \textit{greedy} method to provide a performance guarantee of $(1-1/e)$ (a consequence of string-submodularity). We presented a third heuristic method called \emph{row-swap greedy strategy}, which is guaranteed to perform at least as well as a regular \emph{greedy strategy}, but is computationally more intensive compared to the latter. Our numerical simulations demonstrated that each of the proposed heuristic methods outperforms a commonly used heuristic called \emph{semidefinite relaxation} (SDR).
  
\begin{appendix}
\section{Proof for Proposition~\ref{prop:2Ndom}}
\begin{proof}
  We know that $Tr(\ol{\bs{R}}) = \sum_k^N (4k-2) \delta_k$. If $\bs{R}$ is $2N$-dominant, we can verify that 
  \begin{equation}\label{eq:delta_sumbnd}
  \sum_{k=2}^N \delta_k \leq \frac{1}{4N} Tr(\bs{R}),
  \end{equation}
  Therefore, the following inequalities hold true 
  \begin{equation*}
  \begin{aligned}
  Tr(\ol{\bs{R}}) &= \sum_k^N (4k-2) \delta_k \\
  &\leq (4N-2) \sum_k^N \delta_k \\
  &\leq (4N-2) \left(\frac{1}{4N}\right) Tr(\bs{R})\\
  &\leq Tr(\bs{R}).
  \end{aligned}
  \end{equation*}
  For any given $2N$-dominant Hermitian matrix $\bs{R}$, if $\bs{o}$ is the optimal solution to the UQP, and $r_{ij}$ is the element of $\bs{R}$ at the $i$th row and $j$th column, we can verify the following:
  \begin{equation}\label{eq:optim_bnd}
  \begin{aligned}
  \bs{o}^H \bs{R} \bs{o} &\leq \sum_{i=1}^{N}\sum_{j=1}^N |r_{ij}| = \sum_{i=1}^{N} |r_{ii}| + \sum_i^N \sum_{j = 1, j \neq i}^N |r_{ij}|, \\
   & \leq  \sum_{i=1}^{N} |r_{ii}| + \frac{1}{2N} \sum_{i=1}^{N} |r_{ii}|  = \left(1 + \frac{1}{2N}\right) Tr(\bs{R}).
  \end{aligned}
  \end{equation}
  Also, for any $2N$-dominant Hermitian matrix $\bs{R}$, from Remark~\ref{remark:RbarExp}, (\ref{eq:delta_sumbnd}), and (\ref{eq:optim_bnd}) we can derive the following:  
  \begin{equation*}
  \begin{aligned}
  Tr(\bs{R}) - Tr(\ol{\bs{R}}) &= Tr(\bs{R}) - \sum_{k=2}^N (4k-2) \delta_k\\
  &\geq Tr(\bs{R}) - (4N-2)\sum_{k=2}^N \delta_k \\
  &\geq Tr(\bs{R}) - \left(\frac{4N-2}{4N}\right) Tr(\bs{R})\\
  &= \left(\frac{1}{2N}\right) Tr(\bs{R})\\
  &\geq \left(\frac{1}{2N+1}\right) \bs{o}^H \bs{R} \bs{o}.
  \end{aligned}
  \end{equation*}
  From (\ref{eq:greedy_opt_comp}) and the above result, we can obtain the following:
  \begin{equation*}
  \begin{aligned}
  \bs{g}^H \bs{R} \bs{g} &\geq \left(1 - \frac{1}{e}\right)\bs{o}^H \bs{R} \bs{o} + \frac{1}{e}\left(Tr(\bs{R}) -Tr(\ol{\bs{R}}) \right) \\
  &\geq \left(1 - \frac{1}{e}\right)\bs{o}^H \bs{R} \bs{o} + \frac{1}{e}\left(\frac{1}{2N+1}\right) \bs{o}^H \bs{R} \bs{o}\\
  &= \left( 1 - \frac{1}{e} + \frac{1}{e}\left(\frac{1}{2N+1}\right)\right) \bs{o}^H \bs{R} \bs{o}.
  \end{aligned}
  \end{equation*}
  \qed
  \end{proof}
  
\end{appendix}
  \bibliographystyle{IEEEtran}
  \bibliography{ref}
  
\end{document}